\documentclass[10pt]{amsart}
\usepackage{amsmath}
\usepackage{graphics,graphicx}
\usepackage{color}
\usepackage{caption} 
\usepackage{amsfonts}
\usepackage{epsfig}
\usepackage{verbatim}
\usepackage{amssymb}
\usepackage{latexsym}
\usepackage{amsmath, amsthm, amssymb}

\newtheorem{corollary}{Corollary}[section]

\newtheorem{lemma}[corollary]{Lemma}

\newtheorem{remark}[corollary]{Remark}
\newtheorem{theorem}[corollary]{Theorem}
\newcommand{\mylabel}[1]{\label{#1}
            \ifx\undefined\stillediting
            \else \fbox{$#1$}\fi }
\newcommand{\BE}{\begin{equation}}

\newcommand{\EEQ}{\end{equation}}
\newcommand{\rfb}[1]{\mbox{\rm
   (\ref{#1})}\ifx\undefined\stillediting\else:\fbox{$#1$}\fi}

\newfont{\Blackboard}{msbm10 scaled 1200}

\newfont{\roma}{cmr10 scaled 1200}

\def\CC{\rm \hbox{C\kern-.56em\raise.4ex
         \hbox{$\scriptscriptstyle |$}\kern+0.5 em }}

\newcommand{\half}   {{\frac{1}{2}}}

\newcommand{\mm}    {{\hbox{\hskip 0.5pt}}}

\newcommand{\bluff} {{\hbox{\raise 15pt \hbox{\mm}}}}
\makeatletter
\def\section{\@startsection {section}{1}{\z@}{-3.5ex plus -1ex minus
    -.2ex}{2.3ex plus .2ex}{\large\bf}}
\makeatother
\def\be{\begin{equation}}
\def\ee{\end{equation}}

\def\ds{\displaystyle}
\input{epsf}

\begin{document}
\thispagestyle{empty}
\title[Abstract coupled hyperbolic-parabolic system]{Note on stability of an abstract coupled hyperbolic-parabolic system: singular case}
\author{Ka\"{\i}s Ammari}
\address{LR Analysis and Control of PDEs, LR 22ES03, Department of Mathematics, Faculty of Sciences of Monastir, University of Monastir, Tunisia}
\email{kais.ammari@fsm.rnu.tn}

\author{Farhat Shel}
\address{LR Analysis and Control of PDEs, LR 22ES03, Department of Mathematics, Faculty of Sciences of Monastir, University of Monastir, Tunisia}
\email{farhat.shel@fsm.rnu.tn} 

\author{Zhuangyi Liu}
\address{Department of Mathematics and Statistics, University of Minnesota, Duluth, MN 55812-3000, United States} 
\email{zliu@d.umn.edu}

\begin{abstract} In this paper we try to complete the stability analysis for an abstract system of coupled hyperbolic and parabolic equations
\begin{equation*}
\left\{
\begin{array}{lll}
\ds u_{tt} + Au - A^\alpha w = 0, \\
w_t + A^\alpha u_t + A^\beta w = 0,\\
u(0) = u_0, u_t(0) = u_1, w(0) = w_0,
\end{array}
\right.
\end{equation*} 
where  $A$ is a self-adjoint, positive definite operator on a complex Hilbert space $H$, and $(\alpha, \beta) \in [0,1] \times [0,1]$, which is considered in \cite{Amk}, and after, in \cite{liu1}. Our contribution is to identify a fine scale of polynomial stability of the solution in the region $ S_3: = \left\{(\alpha,\beta) \in [0,1] \times [0,1]; \, \beta < 2\alpha -1 \right\}$ taking into account the presence of a singularity at zero.  
\end{abstract}


\maketitle

\tableofcontents

\section{Introduction} \label{secintro}
\setcounter{equation}{0}

In this paper, we study the stability of the following system\,:   
\be 
\label{eq1} 
u_{tt} + Au - A^\alpha w = 0, 
\ee
\be
\label{eq2} 
w_t + A^\alpha u_t + A^\beta w = 0, 
\ee 
\be 
\label{eq3}
u(0) = u_0, u_t(0) = u_1, w(0) = w_0, 
\ee 

where  $A$ is a self-adjoint, positive definite operator on a complex Hilbert space $H$, and $(\alpha, \beta) \in S_3 = \left\{(a,b) \in [0,1] \times [0,1]; \, b < 2a -1 \right\}$.

By denoting $U = (u,u_t,w)^T, U_0 = (u_0,u_1, w_0)^T$, system \rfb{eq1}-\rfb{eq3} can be written as an abstract linear evolution equation on the space $\mathcal{H} = \mathcal{D} (A^\half) \times H \times H$,

\be
\label{eq4}
\left\{
\begin{array}{ll}
\ds \frac{dU}{dt}(t) = \mathcal{A}_{\alpha,\beta} U(t), \, t  \geq 0, \\
U(0) = U_0,
\end{array}
\right.
\ee
where the operator $\mathcal{A}_{\alpha,\beta} : \mathcal{D}(\mathcal{A}_{\alpha,\beta}) \subset \mathcal{H} \rightarrow \mathcal{H}$ is defined by 

\begin{equation*}
\mathcal{A}_{\alpha,\beta} = \left(
\begin{array}{ccll}
0 & I & 0 \\
- A & 0 & A^\alpha \\
0 & - A^\alpha & - A^\beta
\end{array}
\right),
\end{equation*}
with the domain 
\begin{equation*}
\mathcal{D}(\mathcal{A}_{\alpha,\beta}) = \mathcal{D} (A) \times 
\mathcal{D} (A^{\alpha}) \times 
\mathcal{D} (A^{\alpha}),
\end{equation*}

Firstly, we give a detailed review about the well-posedness of the problem (\ref{eq4}). The operator $\mathcal{A}_{\alpha,\beta}$ is densely defined and dissipative, we will prove that its closure  generates a $C_0$-semigroup of contractions. Using a Lumer-Phillips  theorem
\cite{LuPh61}, it suffices to prove that the adjoint operator $\mathcal{A}^*_{\alpha,\beta}$  is also dissipative.

 The operator $\mathcal{A}^*_{\alpha,\beta}$ is a closed extension of the operator
\begin{equation*}
\mathcal{M}_{\alpha,\beta} = \left(
\begin{array}{ccll}
0 & -I & 0 \\
 A & 0 &- A^\alpha \\
0 &  A^\alpha & - A^\beta
\end{array}
\right),
\end{equation*} 
with domain 
\begin{equation*}
\mathcal{D}(\mathcal{M}_{\alpha,\beta}) =\mathcal{D}(\mathcal{A}_{\alpha,\beta})= \mathcal{D} (A) \times 
\mathcal{D} (A^{\alpha}) \times 
\mathcal{D} (A^{\alpha}).
\end{equation*}
The operator $\mathcal{M}_{\alpha,\beta}$ is densely defined and dissipative, then it is closable and its $\overline{\mathcal{M}}_{\alpha,\beta}$ is also dissipative. To conclude,  it suffices to prove that $\mathcal{A}^*_{\alpha,\beta}=\overline{\mathcal{M}}_{\alpha,\beta}$ . For this we use the following  lemma \cite{ ALMS94, Amk}.

\begin{lemma}\label{l2} \cite{ALMS94}
 We consider on the Hilbert spaces $G$ and $H_1$ the operators
\begin{eqnarray*}
\mathcal{A}:\mathcal{D}(\mathcal{A})\subset G\rightarrow G,\;\;\;\;\;B:\mathcal{D}(B)\subset H_1\rightarrow G,\\
B^*:\mathcal{D}(B^*)\subset G\rightarrow H_1,\;\;\;\;\;\mathcal{C}:\mathcal{D}(\mathcal{C})\subset H_1\rightarrow H_1.
\end{eqnarray*}
and we consider the operator matrix $\mathcal{M}$ on $G\times H_1$ defined by
\begin{equation*}
\mathcal{M}:=\left(
\begin{array}{ccll}
\mathcal{A} & B  \\
-B^* & \mathcal{C} 
\end{array}
\right),\;\;\;\mathcal{D}(\mathcal{M}):=\left(\mathcal{D}(\mathcal{A})\cap  \mathcal{D}(B^*)\right)\times \left(\mathcal{D}(B)\cap \mathcal{D}(\mathcal{C}) \right).  
\end{equation*}
Assume that $\mathcal{C}$ is boundedly invertible, that $\mathcal{B}\in \mathcal{L}\left( \mathcal{D}(\mathcal{C}),G\right) $ and that $\mathcal{C}^{-1}\mathcal{B}^*$ extends to a bounded linear operator (we denote its closure with the same symbol). Then $\mathcal{M}$ is closed if and only if $\mathcal{A}+\mathcal{B}\mathcal{C}^{-1}\mathcal{B}^*$ is closed and one has
\begin{eqnarray*}
\overline{\mathcal{M}}=\left(
\begin{array}{ccll}
I & \mathcal{B}\mathcal{C}^{-1}  \\
 0 & I 
\end{array}
\right)\left(
\begin{array}{ccll}
\overline{\mathcal{A}+\mathcal{B}\mathcal{C}^{-1}\mathcal{B}^*} & 0  \\
 0 & \mathcal{C} 
\end{array}
\right)\left(
\begin{array}{ccll}
I & 0  \\
 -\mathcal{C}^{-1}\mathcal{B}^* & I 
\end{array}
\right)\\
\mathcal{D}\left( \overline{\mathcal{M}}\right) =\left\lbrace \left(
\begin{array}{cll}
u  \\
 v 
\end{array}
\right) \in \mathcal{D}(\overline{\mathcal{A}+\mathcal{B}\mathcal{C}^{-1}\mathcal{B}^*})\times H_1, \;-\mathcal{C}^{-1}\mathcal{B}^*u+v\in \mathcal{D}(\mathcal{C}) \right\rbrace. 
\end{eqnarray*} 
\end{lemma}

 By taking $G=\mathcal{D}(A^{1/2})$, $H_1=H\times H$, $\mathcal{A}=0$, $\mathcal{B}=(-I \; 0):\mathcal{D}(\mathcal{B})=\mathcal{D}(A^\alpha)\times \mathcal{D}(A^\alpha)\subset H\times H \rightarrow \mathcal{D}(A^{1/2})$,   $\mathcal{C}=\left(
\begin{array}{ccll}
0 &- A^\alpha \\
  A^\alpha & - A^\beta
\end{array}
\right)$, with $\mathcal{D}(\mathcal{C})=\mathcal{D}(A^\alpha)\times \mathcal{D}(A^\alpha)$, it appears that $\mathcal{B}^*=\left(
\begin{array}{cll}
-
A  \\
 0 
\end{array}
\right)$, $\mathcal{C}^{-1}\mathcal{B}^*=\left(
\begin{array}{cll}
A^{1+\beta-2\alpha}  \\

 A^{1-\alpha} 
\end{array}
\right)$,  $\mathcal{A}+\mathcal{B}\mathcal{C}^{-1}\mathcal{B}^*=-A^{\beta+1-2\alpha}$ (with domain $\mathcal{D}(A)$), and that $\mathcal{D}\left( \overline{\mathcal{A}+\mathcal{B}\mathcal{C}^{-1}\mathcal{B}^*}\right) =\mathcal{D}(A^{1/2})$. Furthermore, the operators $\mathcal{A}$, $\mathcal{B}$ and $\mathcal{C}$ satisfies the hypothesis in Lemma \ref{l2}, then
 $$\mathcal{D}\left( \overline{\mathcal{M}}_{\alpha,\beta}\right) =\left\lbrace (u,v,\theta)^T \in \mathcal{H},\; -A^{1+\beta-2\alpha}u+v\in \mathcal{D}(A^\alpha),\; -
  A^{1-\alpha}u+\theta\in \mathcal{D}(A^\alpha) \right\rbrace. 
 $$
 But a direct calculation, gives $$\mathcal{D}\left( \mathcal{A}^*_{\alpha,\beta}\right) \subset \left\lbrace (u,v,\theta)^T \in \mathcal{H},\; -A^{1+\beta-2\alpha}u+v\in \mathcal{D}(A^\alpha),\;  -A^{1-\alpha}u+\theta\in \mathcal{D}(A^\alpha) \right\rbrace. 
 $$
Hence $\mathcal{A}^*_{\alpha,\beta}=\overline{\mathcal{M}}_{\alpha,\beta}$ and $\overline{\mathcal{A}}_{\alpha,\beta}$ generates a $C_0$-semigroup of contractions.

\medskip

Now, concerning the stability of the semigroup $e^{t\mathcal{A}_{\alpha,\beta}}$, recall that Ammar-Khodja et al. \cite{Amk}, are firstly proved that for every $\alpha, \beta \geq 0$,  $e^{t\mathcal{A}_{\alpha,\beta} }$ is exponentially stable if and only if\;
$
\max(1-2\alpha, 2\alpha-1)< \beta < 2\alpha.
$
\;Later on, a stability analysis has been performed by J. Hao and Z. Liu in \cite{liu} for $(\alpha,\beta) \in [0,1] \times [0,1]$. They divided the unit square $[0,1] \times [0,1]$ into four regions $S$, $S_1$, $S_2$, $S_3$ where
\begin{eqnarray*}
S&: =& \left\{(\alpha,\beta) \in [0,1] \times [0,1]; \, \max(1-2\alpha, 2\alpha-1)\leq \beta \leq 2\alpha \right\},\\
S_1 &: =& \left\{(\alpha,\beta) \in [0,1] \times [0,1]; \; 0<\beta-2\alpha, \; \alpha\geq 0, \; \frac{1}{2}\leq \beta \leq 1 \right\},\\
S_2 &: =& \left\{(a,b) \in [0,1] \times [0,1];  \; \beta<1-2\alpha, \; \alpha\geq 0, \; 0\leq \beta \leq \frac{1}{2} \right\}
\end{eqnarray*}
and $S_3$ as defined below. We summarize the main results in the following theorem (see also \cite{liu1}).
\begin{theorem}
The semigroup $e^{t\mathcal{A}_{\alpha,\beta}}$ has the following stability properties:

(i) In $S$, it is exponentially stable;

(ii) In $S_1 $, it is polynomially stable of order $\frac{1}{2(\beta-2\alpha)}$;

(iii) In $S_2$, it is polynomially stable of order $\frac{1}{2-2(\beta+2\alpha)}$;

(iv) In $S_3$, it is not asymptotically stable.
\end{theorem} 

To justify the non asymptotic stability in region $S_3$, they proved that $0\in \sigma(\mathcal{A}_{\alpha,\beta})$, where $\sigma(\mathcal{A}_{\alpha,\beta})$ is the spectrum of $\mathcal{A}_{\alpha,\beta}$ . Moreover, it can be shown that  \; $\sigma(\mathcal{A}_{\alpha,\beta})\cap \mathbf{i}\mathbb{R}=\left\{0\right\}$ \;in the region $S_3$.

\medskip
 
The main result of this paper then concerns the precise asymptotic behaviour of the solutions of \rfb{eq1}-\rfb{eq3} for $(\alpha,\beta)$ in the region $S_3$, with initial condition in a special subspace of $\mathcal{D} (\mathcal{A}_{\alpha,\beta})$. Precisely, we will estimate $\|e^{t\mathcal{A}_{\alpha,\beta}}\mathcal{A}_{\alpha,\beta}\left(I-\mathcal{A}_{\alpha,\beta} \right)^{-1} \| $ when $t\longrightarrow \infty$.
Our technique is a special frequency and spectral analysis of the corresponding operator.
 
\section{Stabilization} \label{specanal}
\setcounter{equation}{0}

We will justify that the resolvent has only a singularity (at zero) on the imaginary axis (Lemma \ref{2.3} below), and that $\left( \mathbf{i}\lambda-\mathcal{A}_{\alpha,\beta}\right) ^{-1}$ is bounded outside a neighborhood of zero in $\mathbb{
R}$ (Lemma \ref{2.4} below). then we apply a result due to Batty, Chill and Tomilov (\cite[Theorem 7.6 ]{BCT16}) which relate the decay of $\|e^{t\mathcal{A}_{\alpha,\beta}}\mathcal{A}_{\alpha,\beta}\left(I-\mathcal{A}_{\alpha,\beta} \right)^{-1} \| $ to the growth of $\left( \mathbf{i}\lambda-\mathcal{A}_{\alpha,\beta}\right) ^{-1}$ near zero.

\medskip

So, first we recall the corresponding result
\begin{theorem}{(\cite{BCT16}, Theorem 7.6.)} \label{2.1}
Let $(T(t))_{t\geq 0}$ be a bounded $C_0$-semigroup on a Hilbert space $X$, with generator $\mathcal{A}$. assume that $\sigma(\mathcal{A})\cap \mathbf{i}\mathbb{R}=\left\{0\right\}$, and let $\gamma \geq 1$. The following are equivalent:

(i)\;\; $\|\left( \mathbf{i}s-\mathcal{A}\right)^{-1} \|=\left\{
\begin{array}{ll}
\ds O\,(|s|^{-\gamma}), \;\;\, s\rightarrow 0, \\
\ds O\,(1), \;\;\;\;\;\;\;\;\, |s|\rightarrow \infty,
\end{array}
\right. $

(ii)\; $\|T(t)A^{\gamma}(I-A)^{-\gamma}\|=O\,\left( \frac{1}{t}\right) , \;\;\;\;t\rightarrow \infty,$

(iii)\, $\|T(t)A(I-A)^{-1}\|=O\,\left( \frac{1}{t^{1/\gamma}}\right) , \;\;\;\;t\rightarrow \infty$.
\end{theorem}

Hence, our main result is the following

\begin{theorem} \label{main}
We have the following decay in the region $S_3$
\begin{equation} \label{estimpoly}
\|e^{t\mathcal{A}_{\alpha,\beta}}\mathcal{A}_{\alpha,\beta}\left(I-\mathcal{A}_{\alpha,\beta} \right)^{-1} \|=O\,\left(\frac{1}{t}\right), \;\;\;\;t\rightarrow \infty.
\end{equation}
In particular,  the decay of $e^{t\mathcal{A}_{\alpha,\beta} }\mathcal{A}_{\alpha,\beta}x$ to zero is uniform with respect to $x\in\mathcal{D}(\mathcal{A}_{\alpha,\beta})$.

It follows also that, for every $z\in Ran(\mathcal{A}_{\alpha,\beta})$ we have
\begin{equation*}
\|e^{t\mathcal{A}_{\alpha,\beta}}z \| = O \, \left(\frac{1}{t}\right), \;\;\;\;t\rightarrow \infty.
\end{equation*}
\end{theorem}

\begin{proof}
In view of Theorem \ref{2.1}, the proof is a direct consequence of the following three lemmas.
\begin{lemma}\label{2.3} In the region $S_3$, the operator $\mathcal{A}_{\alpha,\beta}$ satisfies
$$\sigma(\mathcal{A}_{\alpha,\beta})\cap \mathbf{i}\mathbb{R}=\left\{0\right\}.$$
\end{lemma}
\begin{proof}
It has been proved in \cite{liu}, at the third section that $\left\{0\right\}\subset \sigma(\mathcal{A}_{\alpha,\beta})\cap \mathbf{i}\mathbb{R}$.

Conversely, to show that there is no nonzero spectrum point on the imaginary axis, we use a contradiction argument. In fact, let $\lambda \in \mathbb{R}$, $\lambda \neq 0$ such that $\mathbf{i}\lambda \in \sigma(\mathcal{A}_{\alpha,\beta})$. Then, there exists a sequence $(U_n)\subset \mathcal{D}(\mathcal{A}_{\alpha,\beta})$, with $\|U_n\|=1$ for all $n$, such that
\begin{equation}
\lim_{n\rightarrow \infty}\|\left( \mathbf{i}\lambda I-\mathcal{A}_{\alpha,\beta}\right)U_n\|=0\label{spec}
\end{equation}
or there exists a sequence $(U_n)\subset \mathcal{D}(\mathcal{M}_{\alpha,\beta})= \mathcal{D}(\mathcal{A}_{\alpha,\beta})$, with $\|U_n\|=1$ for all $n$, such that
\begin{equation}
\lim_{n\rightarrow \infty}\|\left( \mathbf{i}\lambda I+\mathcal{M}_{\alpha,\beta}\right)U_n\|=0.\label{spec'}
\end{equation} 

Setting $U_n=(u_n,v_n,\theta_n)$, then (\ref{spec}) is equivalent to
\begin{eqnarray}
\mathbf{i}\lambda A^{1/2}u_n-A^{1/2}v_n=o(1),\;\;\;\text{in}\;H, \label{eq6"}\\
\mathbf{i}\lambda v_n+Au_n-A^{\alpha}\theta_n=o(1),\;\;\;\text{in}\;H,\label{eq7"}\\
\mathbf{i}\lambda \theta_n+A^{\alpha}v_n+A^{\beta}\theta_n=o(1),\;\;\;\text{in}\;H.\label{eq8"""}
\end{eqnarray}
First, since 
$$Re\left(\left\langle \left( \mathbf{i}\lambda I-\mathcal{A}_{\alpha,\beta}\right)U_n, U_n \right\rangle _\mathcal{H} \right)=\|A^{\beta/2}\theta_n\|^2 
$$
we obtain
\begin{equation} \label{eq8""}
\lim_{n\rightarrow \infty}\|A^{\beta/2}\theta_n\|=0,
\end{equation}
and in particular
\begin{equation}
\lim_{n\rightarrow \infty}\|\theta_n\|=0. \label{eq9"}
\end{equation}
Second, taking inner product of (\ref{eq6"}) with $\frac{1}{\lambda} A^{1/2}u_n$, (\ref{eq7"}) with $\frac{1}{\lambda}v_n$ and (\ref{eq8"""}) with $\frac{1}{\lambda}\theta_n$, taking into account (\ref{eq8""}) and (\ref{eq9"}), we  get
\begin{eqnarray}
\mathbf{i}\|A^{1/2}u_n\|^2-\frac{1}{\lambda}\left\langle v_n,Au_n \right\rangle =o(1),\label{eq11"}\\
\mathbf{i}\|v_n\|^2+\frac{1}{\lambda}\left\langle Au_n,v_n \right\rangle - \frac{1}{\lambda}\left\langle A^\alpha \theta_n,v_n \right\rangle =o(1),\label{eq12"}\\
\frac{1}{\lambda}\left\langle A^\alpha v_n,\theta_n \right\rangle =o(1).\label{eq13"}
\end{eqnarray}
Then, by combining (\ref{eq11"})-(\ref{eq13"}) and using that $\|U_n\|=1$,  it yields
\begin{equation}
\|A^{1/2}u_n\|=\frac{1}{2}+o(1),\;\;\;\;\|v_n\|=\frac{1}{2}+o(1). \label{eq14"}
\end{equation}
Now, since in the region $S_3$,  $1-\alpha<\frac{1}{2}$ and $\frac{3}{2}-2\alpha+\beta<\frac{1}{2}$, then using the boundedness of $\|A^{1/2}u_n\|$  we get, by interpolation,
\begin{equation}
\|A^{1-\alpha}u_n\|=O(1),\;\;\;\;\|A^{1-2\alpha+\beta}u_n\|=O(1). \label{2.12}
\end{equation}
Then, replacing $\mathbf{i}v_n$ by $\lambda u_n$ in (\ref{eq7"}), due to (\ref{eq6"}) and interpolation, next 
taking the inner product of the obtained equation with $\frac{1}{\lambda}A^{1-2\alpha+\beta}u_n$ to get
\begin{equation}
-\lambda\|A^{\frac{1}{2}-\alpha+\frac{\beta}{2}}u_n\|^2+\frac{1}{\lambda}\|A^{1-\alpha+\frac{\beta}{2}}u_n\|^2- \frac{1}{\lambda}\left\langle A^{\alpha}\theta_n,A^{1-2\alpha+\beta}u_n \right\rangle=o(1). \label{2.13} 
\end{equation}
Taking the inner product of (\ref{eq8"""}) with $\frac{1}{\lambda}A^{1-\alpha}u_n$, we get
\begin{equation}
\mathbf{i}\left\langle \theta_n, A^{1-\alpha}u_n \right\rangle+\frac{1}{\lambda}\left\langle A^{1/2}v_n,A^{1/2}u_n \right\rangle+\frac{1}{\lambda}\left\langle A^{\alpha}\theta_n,A^{1-2\alpha+\beta}u_n \right\rangle=o(1).\label{2.14}
\end{equation}
By (\ref{eq9"}) and (\ref{2.12}), the first term in (\ref{2.13}) converge to zero. Moreover, using (\ref{eq6"}), we can replace $\frac{1}{\lambda}A^{1/2}v_n$ in the second term in (\ref{2.14}) by  $\mathbf{i} A^{1/2}u_n$. Consequently, the sum of (\ref{2.13}) and (\ref{2.14}) yields
\begin{equation}
-\lambda\|A^{\frac{1}{2}-\alpha+\frac{\beta}{2}}u_n\|^2+\frac{1}{\lambda}\|A^{1-\alpha+\frac{\beta}{2}}u_n\|^2+\mathbf{i}\|A^{\frac{1}{2}}u_n\|^2=o(1).
\end{equation}

Hence, $\|A^{\frac{1}{2}}u_n\|^2=o(1)$, which contradict the first estimate in (\ref{eq14"}).

The same approach applied to (\ref{spec'}) leads to the same conclusion without any difficulty.
\end{proof}

\begin{lemma}\label{2.4}
In the region $S_3$,
$$\|\left( \mathbf{i}s-\mathcal{A}_{\alpha,\beta}\right)^{-1} \|=
 O\,(1), \;\;\, s\rightarrow \infty.$$

\end{lemma}
\begin{proof}
It is a direct consequence of Theorem 2.3 in \cite{liu}, since $\limsup\limits_{\lambda\in\mathbb{R},\,\lambda\rightarrow \infty}|\lambda|^{\frac{\beta}{\alpha}}\|\left( \mathbf{i}\lambda-\mathcal{A}_{\alpha,\beta}\right)^{-1}\|<\infty$.
\end{proof}
\begin{lemma}\label{2.5}
In the region $S_3$,
$$\|\left( \mathbf{i}s-\mathcal{A}_{\alpha,\beta}\right)^{-1} \|= O\,(|s|^{-1}), \;\;\, s\rightarrow 0.$$

\end{lemma}
\begin{proof}
By contradiction, suppose that\; $\limsup\limits_{s\in\mathbb{R},\,s\rightarrow 0}\|s\left( \mathbf{i}s I-\mathcal{A}_{\alpha,\beta}\right)^{-1}\|=\infty$. 

Putting $s=\frac{1}{w}$, this is equivalent to $\limsup\limits_{w\in\mathbb{R},\,|w|\rightarrow \infty}\|w^{-1}\left( \mathbf{i}w^{-1} I-\mathcal{A}_{\alpha,\beta}\right)^{-1}\|=\infty$. Then, there exists a sequence $(w_n)$ of real numbers with $|w_n|\rightarrow\infty$, 
 and a sequence $(U_n)\subset \mathcal{D}(\mathcal{A}_{\alpha,\beta})$, with $\|U_n\|=1$ for all $n$ such that
 \begin{equation}\label{eq5}
 \lim_{n\rightarrow \infty}|w_n|\|\left( \mathbf{i}w_n^{-1} I-\mathcal{A}_{\alpha,\beta}\right)U_n\|=0.
\end{equation}  
Setting $U_n=(u_n,v_n,\theta_n)$, then (\ref{eq5}) is rewretten explicitely as follows
\begin{eqnarray}
\mathbf{i}w_n^{-1}|w_n|A^{1/2}u_n-|w_n|A^{1/2}v_n=o(1),\;\;\;\text{in}\;H, \label{eq6}\\
\mathbf{i}w_n^{-1}|w_n|v_n+|w_n|Au_n-|w_n|A^{\alpha}\theta_n=o(1),\;\;\;\text{in}\;H,\label{eq7}\\
\mathbf{i}w_n^{-1}|w_n|\theta_n+|w_n|A^{\alpha}v_n+|w_n|A^{\beta}\theta_n=o(1),\;\;\;\text{in}\;H.\label{eq8}
\end{eqnarray}
Since 
$Re\left(\left\langle |w_n|\left( \mathbf{i}w_n^{-1} I-\mathcal{A}_{\alpha,\beta}\right)U_n, U_n \right\rangle _\mathcal{H} \right)=-|w_n|\|A^{\beta/2}\theta_n\|^2 $, it yields
\begin{equation} \label{eq8"}
\lim_{n\rightarrow \infty}|w_n|^{1/2}\|A^{\beta/2}\theta_n\|=0,
\end{equation}
and in particular
\begin{equation}
\lim_{n\rightarrow \infty}\|\theta_n\|=0. \label{eq9}
\end{equation}
Then, applying $|w_n|^{-1}A^{-\alpha}$ to (\ref{eq8}), we get, (by taking into account (\ref{eq9})),
\begin{equation} \label{eq10}
\lim_{n\rightarrow \infty}\|v_n\|=0.
\end{equation}
Now, taking inner product of (\ref{eq6}) with $A^{1/2}u_n$, (\ref{eq7}) with $v_n$ and (\ref{eq8}) with $\theta_n$, taking into account (\ref{eq8"}),  (\ref{eq9}) and (\ref{eq10}), we  have
\begin{eqnarray}
\mathbf{i}w_n^{-1}|w_n|\|A^{1/2}u_n\|^2-|w_n|\left\langle v_n,Au_n \right\rangle =o(1),\label{eq11}\\
|w_n|\left\langle Au_n,v_n \right\rangle - |w_n|\left\langle A^\alpha \theta_n,v_n \right\rangle =o(1),\label{eq12}\\
|w_n|\left\langle A^\alpha \theta_n,v_n \right\rangle =o(1).\label{eq13}
\end{eqnarray}
By combining (\ref{eq11})-(\ref{eq13}), it yields
\begin{equation}
\lim_{n\rightarrow \infty}\|A^{1/2}u_n\|=0. \label{eq14}
\end{equation}
The promised contradiction follows from (\ref{eq9}), (\ref{eq10}) and (\ref{eq14}). Thus, the proof of Lemma \ref{2.5} is completed.
\end{proof}
Which ends the proof of theorem.
\end{proof}
\begin{remark}
In $S_3$, the semigroup $e^{t\mathcal{A}_{\alpha,\beta}}$ is of Gevrey class $\delta >\frac{\alpha}{\beta}$ and in particular, it is infinitely differentiable \cite{liu}. Thus, for $z \in Ran(\mathcal{A}_{\alpha,\beta})$, we not only has polynomial stability, but also have instantaneous smoothness.
\end{remark}

\section{Application}

As application of Theorem \ref{main}, we consider the following thermoplate system:

\begin{equation*}
\left\{
\begin{array}{lll}
\ds u_{tt} + \Delta^2 u - (-\Delta)^\alpha w = 0, \Omega \times (0,+\infty),\\
w_t + (- \Delta)^\alpha u_t - \Delta w = 0, \, \Omega \times (0,+\infty),\\
u = \Delta u = w = 0, \, \Gamma \times (0,+\infty),\\
u(x,0) = u_0(x), u_t(x,0) = u_1(x), w(x,0) = w_0(x), \, \Omega,
\end{array}
\right.
\end{equation*} 

where $\Omega$ be a bounded domain in $\mathbb{R}^n$ with smooth boundary $\Gamma$, $\beta = \frac{1}{2}$ and  $3/4 < \alpha \leq 1.$

\medskip

Here, $H = L^2(\Omega), A = \Delta^2$ avec $\mathcal{D}(A) = \left\{u \in H^2(\Omega) \cap H^1_0(\Omega); \, \Delta u = 0\;\text{on}\; \Gamma\right\}$.

\medskip

Then, according to Theorem \ref{main} the corresponding semigroup satisfies the estimate (\ref{estimpoly}) for all $\alpha \in (3/4,1].$


\begin{thebibliography}{99}

\bibitem{Amk} F. Ammar-Khodja, A. Bader and A. Benabdallah, Dynamic stabilization of systems via decoupling techniques, {\em ESAIM Control Optim. Calc. Var.}, {\bf 4} (1999), 577--593.



\bibitem{ALMS94} F. V. Atkinson, H. Langer, R. Mennicken and A. A. Shkalikov, The essential spectrum of some matrix operators. Math. Nachr.,  {\bf 167} (1994), 5--20.

\bibitem{BCT16}  C. J. K. Batty, R. Chill and Y. Tomilov, Fine scales of decays of operator semigroups. J. Eur. Math. Soc., {\bf 18} (2016), 853--929.

stability of linear dynamical systems in Hilbert space, {\em Ann.
Differential Equations}, {\bf 1} (1985), 43--56.

\bibitem{liu1} J. Hao and Z. Liu, Stability of an abstract system of coupled hyperbolic and parabolic equations, {\em Z. Angew. Math. Phys.}, {\bf 64} (2013), 1145--1159. 

\bibitem{liu} J. Hao, Z. Liu and J. Yong, Regularity analysis for an abstract system of coupled hyperbolic and parabolic equations, {\em Journal of Differential Equations}, {\bf 259} (2015), 4763-4798. 

\bibitem{LuPh61} G. Lumer and R. S. Phillips, Dissipative operators in a Banach space, {\em Pacific Journal of Mathematics}, {\bf 11} (1961), 679-698. 


\end{thebibliography}
\end{document}